\documentclass[11pt]{article}

\usepackage[margin=1in]{geometry}
\usepackage{amsmath,amssymb,amsthm,mathtools}
\usepackage{booktabs}
\usepackage{microtype}
\usepackage[hidelinks]{hyperref}

\newtheorem{theorem}{Theorem}[section]
\newtheorem{proposition}[theorem]{Proposition}

\newtheorem{corollary}[theorem]{Corollary}
\newtheorem{remark}[theorem]{Remark}

\newcommand{\D}{\mathrm D}
\newcommand{\R}{\mathcal R}
\newcommand{\M}{\mathcal M}
\newcommand{\E}{\mathcal E}

\title{Stationary Grading and Heat-Transported Moving Boundaries:\\
From the Cubic Case to Two Fourth-Order Branches}

\author{Gerardo Hern\'andez-del-Valle
\thanks{Centro de Estudios Monetarios Latinoamericanos (CEMLA), Mexico City,
Mexico. E-mail: \texttt{ghernandez@cemla.org}}}

\date{}

\begin{document}

\maketitle

\begin{abstract}
We study zeros transported by the heat semigroup from stationary solutions of
the sparse operators $L_n=c_{n,0}\D^n+c_{0,n-1}x^{n-1}$.  The starting point
is the exact cubic boundary obtained by Hern\'andez-del-Valle and
Guerra-Polania~\cite{HerGuerra}.  We show that the hypergeometric function used there is the
residue-$1$ component of a naturally graded stationary solution space, and
that heat transport preserves this grading after the scaling $x=t^ny$.  This
identifies the cubic boundary as the transported zero of a specific stationary
mode.  We then study $L_4=c_{4,0}\D^4+c_{0,3}x^3$.  Its $6\times6$
boundary-jet determinant leads to the universal scaling
\[
F(t)=\rho\left(\frac t2\right)^4
\Phi\!\left(\rho^2\left(\frac t2\right)^7\right),
\qquad \rho=\frac{c_{0,3}}{c_{4,0}}.
\]
The resulting nonlinear equation for $\Phi$ is singular at the origin.
After an explicit desingularization, its initial compatibility condition
factors as
\[
(5\Phi'(0)-1)(6\Phi'(0)-1)=0.
\]
Consequently there are exactly two normalized graded analytic boundary germs. We
place the desingularized equation in Briot--Bouquet form, prove local
existence and uniqueness of both branches, and give a triangular recursion
for all their Taylor coefficients.  Finally, we prove that the two branches
are generated respectively by the residue-$1$ and residue-$3$ stationary
solutions.  The quadratic and cubic cases provide the established baseline;
the fourth-order problem is the first point at which distinct stationary modes
produce distinct analytic branches.  We do not claim a general branch-count
theorem.
\end{abstract}



%%%%%%%%%%%%%%%%%%%%%%%%%%%%%%%%%%%%%%%%%%%%%%%%%%%%%%%%%%%%
\section{Introduction}
%%%%%%%%%%%%%%%%%%%%%%%%%%%%%%%%%%%%%%%%%%%%%%%%%%%%%%%%%%%%

The starting point of the present work can be traced to the cubic-boundary
construction in \cite{HerGuerra}.  There a third-order stationary equation was
transported by the heat semigroup and combined with moving-boundary identities
to obtain an exact cubic absorbing boundary.  The calculation was successful,
but the role of the particular stationary solution and the grading underlying
its heat transport were not isolated.  The essential seed of the present
mechanism was already present, although it was not formulated in these terms.

Our first purpose is to make that mechanism explicit.  The stationary
equation associated with $L_n=c_{n,0}\D^n+c_{0,n-1}x^{n-1}$ has a coefficient
recursion with step $2n-1$.  Under heat transport, the scaling $x=t^ny$
preserves the same congruence structure.  In the cubic case, this identifies
the hypergeometric function of \cite[Section~6]{HerGuerra} as the residue-$1$
stationary mode and the exact cubic boundary as its transported zero.  The
other stationary mode compatible with absorption at the origin begins with
$x^2$ and produces square-root behavior.

The fourth-order problem then reveals why this interpretation matters.  Two
different stationary modes generate two different graded analytic boundary
branches.  To establish this, we combine stationary grading with the
compatibility-matrix framework of \cite{HernandezMatrices}.  A commuting
Heat--Airy lift gives a complementary organization of the same boundary-jet
information; see \cite{HernandezHeatAiry}.

Consider
\begin{equation}
L_4=c_{4,0}\D^4+c_{0,3}x^3,
\qquad c_{4,0}c_{0,3}\ne0.
\label{eq:L4}
\end{equation}
The heat transport of \eqref{eq:L4}, combined with an absorbing boundary
$x=F(t)$, produces a $6\times6$ homogeneous system for the boundary jet.
The vanishing of its determinant is a third-order nonlinear equation for
$F$.  Direct expansion obscures its structure.  The decisive observation is
that the determinant is graded: after the normalization
\[
F(t)=\rho(t/2)^4\Phi(z),
\qquad z=\rho^2(t/2)^7,
\qquad \rho=c_{0,3}/c_{4,0},
\]
all dependence on the two nonzero operator coefficients disappears.

The universal equation remains singular at $z=0$.  This singularity is not
an inconvenience to be discarded: it contains the branching mechanism.
Writing $\Phi=1+zu$ and introducing three Euler-shifted variables produces an
analytic differential-algebraic system.  Its initial equation is
\[
484(5u(0)-1)(6u(0)-1)=0.
\]
Thus the two first corrections $1/5$ and $1/6$ emerge algebraically from the
compatibility determinant.  Each choice then generates one and only one
analytic Taylor series.

Our purpose is not to infer a branch-count formula from three low-order cases.
Rather, we give a common language for the settled quadratic case, reinterpret
the published cubic construction, and prove the fourth-order result on firm
analytic ground.  The emerging pattern is best regarded as a theory in
progress.

%%%%%%%%%%%%%%%%%%%%%%%%%%%%%%%%%%%%%%%%%%%%%%%%%%%%%%%%%%%%
\section{Stationary grading and heat transport}
%%%%%%%%%%%%%%%%%%%%%%%%%%%%%%%%%%%%%%%%%%%%%%%%%%%%%%%%%%%%

Set $\rho=c_{0,n-1}/c_{n,0}$ and consider
\begin{equation}
\phi^{(n)}+\rho x^{n-1}\phi=0.
\label{eq:generalstationary}
\end{equation}
If $\phi(x)=\sum_{k\ge0}a_kx^k$, coefficient comparison gives
\begin{equation}
a_{k+2n-1}
=-\rho\frac{(k+n-1)!}{(k+2n-1)!}a_k,
\qquad k\ge0.
\label{eq:generalrecurrence}
\end{equation}
Thus the stationary solution space is graded by residue classes modulo
$M=2n-1$.

The heat action on a monomial is
\begin{equation}
P_tx^m
=\sum_{j=0}^{\lfloor m/2\rfloor}
\frac{m!}{(m-2j)!j!}\left(\frac t2\right)^jx^{m-2j}.
\label{eq:generalheatmonomial}
\end{equation}
After setting $x=t^ny$, a term obtained by applying $j$ heat contractions to
$x^m$ carries the power
\begin{equation}
t^{\,nm-(2n-1)j}.
\label{eq:heatgrade}
\end{equation}
Its exponent modulo $2n-1$ depends only on the stationary residue class of
$m$.  This elementary identity is the grading mechanism used below.

\begin{remark}
Equation \eqref{eq:heatgrade} is not a general theorem about the number or
regularity of moving-boundary branches.  Whether a stationary mode produces
an analytic, Puiseux, or more singular zero must still be determined in each
case.
\end{remark}

%%%%%%%%%%%%%%%%%%%%%%%%%%%%%%%%%%%%%%%%%%%%%%%%%%%%%%%%%%%%
\section{The quadratic and cubic baselines}
%%%%%%%%%%%%%%%%%%%%%%%%%%%%%%%%%%%%%%%%%%%%%%%%%%%%%%%%%%%%

For $n=2$, the absorbing stationary mode of
$\phi''+\rho x\phi=0$ begins with $x$.  Its heat-transported zero is the
familiar exact quadratic boundary
\[
F(t)=\rho\left(\frac t2\right)^2.
\]
This settled case serves only as a baseline.

For $n=3$, let
\begin{equation}
L_3=c_{3,0}\D^3+c_{0,2}x^2,
\qquad
\rho=\frac{c_{0,2}}{c_{3,0}}.
\label{eq:L3}
\end{equation}
Then $L_3\phi=0$ becomes
\begin{equation}
\phi'''+\rho x^2\phi=0,
\label{eq:L3stationary}
\end{equation}
and
\begin{equation}
a_{k+5}=-\rho\frac{(k+2)!}{(k+5)!}a_k.
\label{eq:L3recurrence}
\end{equation}
The absorbing condition $\phi(0)=0$ leaves two stationary directions.  The
residue-$1$ solution normalized by $\phi_{[1]}'(0)=1$ is
\begin{align}
\phi_{[1]}(x)
&=x-\frac{\rho}{120}x^6+\frac{\rho^2}{118800}x^{11}+\cdots \notag\\
&=x\,{}_0F_2\!\left(;\frac45,\frac65;-\frac{\rho x^5}{125}\right).
\label{eq:L3phi1}
\end{align}
Up to normalization and the identification $b=-\rho$, this is the
hypergeometric solution used in \cite[Example~6.2]{HerGuerra}.

\begin{proposition}[The cubic stationary mode]
\label{prop:cubicmode}
Let $v_{[1]}=P_t\phi_{[1]}$.  Then
\begin{equation}
v_{[1]}\!\left(t,\frac{\rho}{8}t^3\right)=0,
\label{eq:exactcubiczero}
\end{equation}
and therefore
\begin{equation}
F(t)=\frac{\rho}{8}t^3
=\rho\left(\frac t2\right)^3.
\label{eq:cubicboundary}
\end{equation}
Thus the exact cubic boundary is the heat-transported zero of the residue-$1$
stationary solution.
\end{proposition}

\begin{proof}
With $b=-\rho$, equations \eqref{eq:L3stationary} and
\eqref{eq:cubicboundary} are precisely the setting of Theorem~6.1 in
\cite{HerGuerra}.  The new point is the stationary-mode identification:
\eqref{eq:L3recurrence} shows that the function used there is supported on the
single residue class $1\pmod5$.  Equivalently,
\eqref{eq:generalheatmonomial} gives
\[
v_{[1]}(t,t^3y)=t^3W_1(t^5,y),
\qquad
W_1(0,y)=y-\frac{\rho}{8}.
\]
The exact vanishing in \eqref{eq:exactcubiczero} is the published cubic
theorem, now interpreted through the preserved grading.
\end{proof}

The second absorbing stationary direction begins with
\begin{equation}
\phi_{[2]}(x)=\frac{x^2}{2}-\frac{\rho}{420}x^7+\cdots.
\label{eq:L3phi2}
\end{equation}
Since $P_t(x^2/2)=(x^2+t)/2$, its small zeros satisfy
$x\sim\pm\sqrt{-t}$.  Hence there is only one residue-mode boundary in the
analytic cubic class.  This distinction was not needed for the derivation in
\cite{HerGuerra}, but it becomes decisive at order four.

%%%%%%%%%%%%%%%%%%%%%%%%%%%%%%%%%%%%%%%%%%%%%%%%%%%%%%%%%%%%
\section{The fourth-order compatibility matrix}
%%%%%%%%%%%%%%%%%%%%%%%%%%%%%%%%%%%%%%%%%%%%%%%%%%%%%%%%%%%%

Let
\[
P_t=\exp\!\left(\frac t2\D^2\right).
\]
The canonical commutator gives
\[
P_txP_t^{-1}=x+t\D,
\]
and hence
\begin{equation}
N_4:=P_tL_4P_t^{-1}
=c_{4,0}\D^4+c_{0,3}(x+t\D)^3.
\label{eq:N4}
\end{equation}
The normal-ordered cubic is
\begin{equation}
(x+t\D)^3
=x^3+3tx+(3tx^2+3t^2)\D+3t^2x\D^2+t^3\D^3.
\label{eq:normalcubic}
\end{equation}

Let $v_t=\tfrac12v_{xx}$ and suppose
\[
v(t,F(t))=0.
\]
We write
\[
v^{(k)}=\partial_x^kv(t,F(t)).
\]
Successive differentiation of the boundary condition gives homogeneous
relations among the spatial jets.  Combining the first three such relations
with the restrictions of $N_4v$, $\D N_4v$, and $\D^2N_4v$ produces six
equations for
\begin{equation}
\mathbf V=(v^{(6)},v^{(5)},v^{(4)},v^{(3)},v^{(2)},v^{(1)})^{\mathsf T}.
\label{eq:jet}
\end{equation}
Dividing the transport rows by $c_{4,0}$ and temporarily taking $\rho=1$
gives the normalized compatibility matrix
\begin{equation}
\M[F]=
\begin{pmatrix}
1&t^3&3t^2F&3t(F^2+3t)&F^3+15tF&6(F^2+t^2)\\
\frac12&3F'&6(F')^2&6F''+4(F')^3&12F'F''&4F'''\\
0&1&t^3&3t^2F&3t(F^2+2t)&F^3+9tF\\
0&0&1&4F'&4(F')^2&4F''\\
0&0&1&t^3&3t^2F&3t(F^2+t)\\
0&0&0&0&1&2F'
\end{pmatrix}.
\label{eq:matrix}
\end{equation}
The odd rows are the restrictions of $\D^2N_4v$, $\D N_4v$, and
$N_4v$, respectively; the even rows are the first three heat-boundary
relations, in the normalization used in \eqref{eq:matrix}.

\begin{proposition}[Compatibility condition]
\label{prop:det}
If the boundary jet \eqref{eq:jet} is nonzero, then a necessary local
compatibility condition is
\begin{equation}
\det\M[F]=0.
\label{eq:det}
\end{equation}
For general $\rho$, the transport rows are restored before taking the
determinant; equivalently, the coefficient dependence is recovered by the
scaling in Section~\ref{sec:scaling}.
\end{proposition}

\begin{proof}
The six boundary relations have the homogeneous form
$\M[F]\mathbf V=0$.  A nonzero boundary jet therefore requires the
coefficient matrix to be singular.
\end{proof}

\begin{remark}
Equation \eqref{eq:det} is a local boundary-jet condition.  It does not assert
the existence of a global solution $v$.  The question addressed here is the
local compatibility of the transported equation with the absorbing boundary.
\end{remark}

%%%%%%%%%%%%%%%%%%%%%%%%%%%%%%%%%%%%%%%%%%%%%%%%%%%%%%%%%%%%
\section{Universal scaling}
\label{sec:scaling}
%%%%%%%%%%%%%%%%%%%%%%%%%%%%%%%%%%%%%%%%%%%%%%%%%%%%%%%%%%%%

The determinant in \eqref{eq:det} is invariant under a weighted grading.
Motivated by its lowest-order balance, define
\begin{equation}
F(t)=\rho\left(\frac t2\right)^4\Phi(z),
\qquad
z=\rho^2\left(\frac t2\right)^7,
\qquad
\rho=\frac{c_{0,3}}{c_{4,0}}.
\label{eq:scaling}
\end{equation}
Let
\[
\theta=z\frac{d}{dz}
\]
and introduce
\begin{equation}
A=(4+7\theta)\Phi,
\qquad
B=(3+7\theta)A,
\qquad
C=(2+7\theta)B.
\label{eq:ABC}
\end{equation}
Then
\begin{equation}
F'=\frac{\rho t^3}{16}A,
\qquad
F''=\frac{\rho t^2}{16}B,
\qquad
F'''=\frac{\rho t}{16}C.
\label{eq:Fderivatives}
\end{equation}

Direct substitution of \eqref{eq:scaling}--\eqref{eq:Fderivatives} into
\eqref{eq:det}, followed by division by the common nonzero monomial factor,
gives the following coefficient-free equation.

\begin{proposition}[Universal equation]
\label{prop:universal}
The scaled boundary function $\Phi$ satisfies
\begin{equation}
\E[\Phi]=0,
\label{eq:universal}
\end{equation}
where
\begin{align}
\E[\Phi]={}&240+8C-72B+3B^2+48A-2AC \notag\\
&+z\bigl[-512B+192AB+384A^2-48A^2B-48A^3+4A^3B \notag\\
&\hspace{2.9em}+192\Phi B-768\Phi A+48\Phi A^2
+240\Phi^2-24\Phi^2B+12\Phi^2A\bigr] \notag\\
&+z^2\bigl[512A^3-192A^4+24A^5-A^6
-1536\Phi A^2+384\Phi A^3-24\Phi A^4 \notag\\
&\hspace{2.9em}+1536\Phi^2A-192\Phi^2A^2
-512\Phi^3-32\Phi^3A+4\Phi^3A^2+36\Phi^4\bigr].
\label{eq:Eexplicit}
\end{align}
In particular, the equation is independent of $c_{4,0}$ and $c_{0,3}$.
\end{proposition}

At $z=0$, equation \eqref{eq:universal} reduces to
\[
240(\Phi(0)-1)^2=0.
\]
The normalized solution therefore satisfies $\Phi(0)=1$.  Notice, however,
that the coefficient of the highest Euler derivative $C$ is
\[
2(4-A),
\]
which vanishes at $z=0$ because $A(0)=4$.  Thus \eqref{eq:universal} is
singular precisely at the point where the initial condition is imposed.

%%%%%%%%%%%%%%%%%%%%%%%%%%%%%%%%%%%%%%%%%%%%%%%%%%%%%%%%%%%%
\section{Desingularization and algebraic branching}
%%%%%%%%%%%%%%%%%%%%%%%%%%%%%%%%%%%%%%%%%%%%%%%%%%%%%%%%%%%%

Write
\begin{equation}
\Phi=1+zu,
\qquad
A=4+za,
\qquad
B=12+zb,
\qquad
C=24+zc.
\label{eq:desub}
\end{equation}
The definitions \eqref{eq:ABC} become
\begin{equation}
7\theta u=a-11u,
\qquad
7\theta a=b-10a,
\qquad
7\theta b=c-9b.
\label{eq:eulerchain}
\end{equation}

Substitution into \eqref{eq:Eexplicit} produces an exact factor $z^2$.
After dividing by that factor, the remaining equation is
\begin{align}
0=\R(z,u,a,b,c):={}&484-88b+3b^2+396a-2ac \notag\\
&+z(-144a^2-1584u+144ub-360ua) \notag\\
&+z^2(-188a^2+48ua^2+1560u^2-24u^2b+12u^2a) \notag\\
&+z^3(4a^3b+396ua^2-432u^3) \notag\\
&+z^4(24a^4-180u^2a^2+36u^4) \notag\\
&+z^5(-24ua^4+4u^3a^2)-z^6a^6.
\label{eq:R}
\end{align}

\begin{proposition}[Branch equation]
\label{prop:branch}
Every solution analytic at the origin with $\Phi(0)=1$ satisfies
\begin{equation}
(5\Phi'(0)-1)(6\Phi'(0)-1)=0.
\label{eq:branch}
\end{equation}
\end{proposition}

\begin{proof}
Put $u_0=u(0)=\Phi'(0)$.  Evaluation of \eqref{eq:eulerchain} at the
origin gives
\[
a_0=11u_0,
\qquad b_0=110u_0,
\qquad c_0=990u_0.
\]
The constant term of \eqref{eq:R} is therefore
\[
484-5324u_0+14520u_0^2
=484(5u_0-1)(6u_0-1).
\]
This proves \eqref{eq:branch}.
\end{proof}

The two possible initial triples are
\begin{equation}
\mathbf y_0^{(1)}=
\left(\frac15,\frac{11}{5},22\right),
\qquad
\mathbf y_0^{(2)}=
\left(\frac16,\frac{11}{6},\frac{55}{3}\right),
\label{eq:initialtriples}
\end{equation}
where $\mathbf y=(u,a,b)$.

%%%%%%%%%%%%%%%%%%%%%%%%%%%%%%%%%%%%%%%%%%%%%%%%%%%%%%%%%%%%
\section{The local two-branch theorem}
%%%%%%%%%%%%%%%%%%%%%%%%%%%%%%%%%%%%%%%%%%%%%%%%%%%%%%%%%%%%

Equation \eqref{eq:R} is affine in $c$:
\[
\R=-2ac+Q(z,u,a,b).
\]
At both points in \eqref{eq:initialtriples}, $a\ne0$.  Hence
\begin{equation}
c=\frac{Q(z,u,a,b)}{2a}
\label{eq:csolve}
\end{equation}
is analytic in a neighborhood of either initial point.  Combining
\eqref{eq:csolve} and \eqref{eq:eulerchain} gives
\begin{equation}
\theta\mathbf y=\mathbf H(z,\mathbf y),
\label{eq:BB}
\end{equation}
where
\begin{equation}
\mathbf H(z,u,a,b)=\frac17
\left(
a-11u,
b-10a,
\frac{Q(z,u,a,b)}{2a}-9b
\right).
\label{eq:H}
\end{equation}
This is a Briot--Bouquet system.

\begin{theorem}[Two analytic branches]
\label{thm:two}
The universal equation \eqref{eq:universal} has exactly two solution germs
analytic at $z=0$ and normalized by $\Phi(0)=1$.  They are characterized by
\[
\Phi_1'(0)=\frac15,
\qquad
\Phi_2'(0)=\frac16.
\]
Once one of these values is chosen, all subsequent Taylor coefficients are
uniquely determined.
\end{theorem}

\begin{proof}
Proposition~\ref{prop:branch} shows that an analytic normalized germ must
begin at one of the two points in \eqref{eq:initialtriples}.  At the first
point, the eigenvalues of
$D_{\mathbf y}\mathbf H(0,\mathbf y_0^{(1)})$ are
\[
-\frac{11}{7},\qquad \frac17,\qquad -\frac{10}{7}.
\]
At the second point, the eigenvalues of
$D_{\mathbf y}\mathbf H(0,\mathbf y_0^{(2)})$ are
\[
-\frac{11}{7},\qquad -\frac17,\qquad -\frac{12}{7}.
\]
None is a positive integer.  The analytic Briot--Bouquet theorem therefore
gives a unique analytic solution of \eqref{eq:BB} through each initial point.
Recovering $\Phi=1+zu$ gives precisely two normalized analytic germs.
\end{proof}

\begin{remark}
The positive eigenvalue $1/7$ in the first branch is nonintegral.  It does
not obstruct the analytic Taylor solution, but it suggests the possible
presence of fractional-power solution directions outside the analytic class.
We do not pursue those solutions here.
\end{remark}

%%%%%%%%%%%%%%%%%%%%%%%%%%%%%%%%%%%%%%%%%%%%%%%%%%%%%%%%%%%%
\section{A triangular coefficient recursion}
%%%%%%%%%%%%%%%%%%%%%%%%%%%%%%%%%%%%%%%%%%%%%%%%%%%%%%%%%%%%

The proof above establishes convergence, while the desingularized equation
also supplies an efficient exact algorithm.  Write
\begin{equation}
u(z)=\sum_{n=0}^{\infty}q_nz^n,
\qquad
\Phi(z)=1+\sum_{m=1}^{\infty}A_mz^m,
\qquad A_m=q_{m-1}.
\label{eq:series}
\end{equation}
The Euler chain \eqref{eq:eulerchain} gives, coefficient by coefficient,
\begin{equation}
a_n=(7n+11)q_n,
\quad
b_n=(7n+10)(7n+11)q_n,
\quad
c_n=(7n+9)(7n+10)(7n+11)q_n.
\label{eq:abc-coeff}
\end{equation}

For $n\ge1$, the coefficient of $z^n$ in \eqref{eq:R} is linear in the
new coefficient $q_n$.  It has the form
\begin{equation}
\Lambda_n^{(j)}q_n+P_n^{(j)}(q_0,\ldots,q_{n-1})=0,
\label{eq:triangular}
\end{equation}
where $j=1,2$ labels the selected branch and the polynomial
$P_n^{(j)}$ is obtained by coefficient extraction from \eqref{eq:R} using
only previously known coefficients.  The two divisors are
\begin{align}
\Lambda_n^{(1)}
&=-\frac{22}{5}(7n+11)(7n-1)(7n+10),
\label{eq:lambda1}\\
\Lambda_n^{(2)}
&=-\frac{11}{3}(7n+11)(7n+1)(7n+12).
\label{eq:lambda2}
\end{align}
Neither vanishes for an integer $n\ge1$.  Thus
\begin{equation}
q_n=-\frac{P_n^{(j)}(q_0,\ldots,q_{n-1})}{\Lambda_n^{(j)}}.
\label{eq:recursion}
\end{equation}
After the single quadratic choice at $n=0$, every step is only a linear
division.

The first coefficients are displayed in Table~\ref{tab:coefficients}.
\begin{table}[ht]
\centering
\begin{tabular}{c@{\qquad}cc}
\toprule
$m$ & $A_m^{(1)}$ & $A_m^{(2)}$\\
\midrule
0 & $1$ & $1$\\
1 & $1/5$ & $1/6$\\
2 & $-1/15$ & $-1/24$\\
3 & $-17/3900$ & $-11/1440$\\
4 & $171/6500$ & $761/51840$\\
5 & $-389537/21674250$ & $-1333/222720$\\
\bottomrule
\end{tabular}
\caption{The first universal coefficients in
$\Phi_j(z)=\sum_{m\ge0}A_m^{(j)}z^m$.}
\label{tab:coefficients}
\end{table}

Consequently,
\begin{align}
\Phi_1(z)
&=1+\frac z5-\frac{z^2}{15}-\frac{17z^3}{3900}
+\frac{171z^4}{6500}-\frac{389537z^5}{21674250}+O(z^6),
\label{eq:phi1}\\
\Phi_2(z)
&=1+\frac z6-\frac{z^2}{24}-\frac{11z^3}{1440}
+\frac{761z^4}{51840}-\frac{1333z^5}{222720}+O(z^6).
\label{eq:phi2}
\end{align}

%%%%%%%%%%%%%%%%%%%%%%%%%%%%%%%%%%%%%%%%%%%%%%%%%%%%%%%%%%%%
\section{The two physical boundaries}
%%%%%%%%%%%%%%%%%%%%%%%%%%%%%%%%%%%%%%%%%%%%%%%%%%%%%%%%%%%%

Returning to the original variables, define
\begin{equation}
F_j(t)=\rho\left(\frac t2\right)^4
\Phi_j\!\left(\rho^2\left(\frac t2\right)^7\right),
\qquad j=1,2.
\label{eq:Fj}
\end{equation}
The universal coefficients give the graded expansions
\begin{equation}
F_j(t)=\sum_{m=0}^{\infty}
A_m^{(j)}\rho^{2m+1}\left(\frac t2\right)^{7m+4}.
\label{eq:Fseries}
\end{equation}
In particular,
\begin{align}
F_1(t)
&=\frac{\rho}{16}t^4
+\frac{\rho^3}{10240}t^{11}
-\frac{\rho^5}{3932160}t^{18}+O(t^{25}),
\label{eq:F1}\\
F_2(t)
&=\frac{\rho}{16}t^4
+\frac{\rho^3}{12288}t^{11}
-\frac{\rho^5}{6291456}t^{18}+O(t^{25}).
\label{eq:F2}
\end{align}
Both boundaries have the same quartic leading behavior, but the determinant
distinguishes them beginning at order eleven.

\begin{corollary}
For every fixed nonzero $\rho$, the fourth-order compatibility condition has
exactly two boundary germs analytic at $t=0$ within the normalization
$F(0)=F'(0)=F''(0)=F'''(0)=0$ and
$[t^4]F=\rho/16$, provided the boundary is sought in the graded analytic
class \eqref{eq:Fseries}.
\end{corollary}

%%%%%%%%%%%%%%%%%%%%%%%%%%%%%%%%%%%%%%%%%%%%%%%%%%%%%%%%%%%%
\section{Identification of the underlying stationary solutions}
\label{sec:stationary}
%%%%%%%%%%%%%%%%%%%%%%%%%%%%%%%%%%%%%%%%%%%%%%%%%%%%%%%%%%%%

The two boundary branches can be identified directly in physical space.
After division by $c_{4,0}$, the stationary equation is
\begin{equation}
\phi^{(4)}+\rho x^3\phi=0.
\label{eq:stationary}
\end{equation}
If $\phi(x)=\sum_{k\ge0}a_kx^k$, coefficient comparison gives
\begin{equation}
a_{k+7}=-\rho\frac{(k+3)!}{(k+7)!}a_k,
\qquad k\ge0.
\label{eq:stationary-rec}
\end{equation}
Thus the solution space splits into four residue classes modulo seven.  We
shall need the solutions normalized by
\begin{align}
\phi_{[1]}(0)&=0,&
\phi_{[1]}'(0)&=1,&
\phi_{[1]}''(0)&=\phi_{[1]}'''(0)=0,
\label{eq:phi1data}\\
\phi_{[3]}(0)&=\phi_{[3]}'(0)=\phi_{[3]}''(0)=0,&
\phi_{[3]}'''(0)&=1.
\label{eq:phi3data}
\end{align}
Their series begin
\begin{align}
\phi_{[1]}(x)
&=x-\frac{\rho}{1680}x^8
+\frac{\rho^2}{55036800}x^{15}+\cdots,
\label{eq:phi1stationary}\\
\phi_{[3]}(x)
&=\frac{x^3}{6}-\frac{\rho}{30240}x^{10}+\cdots.
\label{eq:phi3stationary}
\end{align}
The factorial decay in \eqref{eq:stationary-rec} shows that these are entire
functions of order strictly less than two.  Consequently their heat
transforms
\begin{equation}
v_{[r]}(t,x)=P_t\phi_{[r]}(x),
\qquad r=1,3,
\label{eq:heatmodes}
\end{equation}
are analytic for $(t,x)$ sufficiently close to $(0,0)$.  Moreover,
\[
N_4v_{[r]}=P_tL_4\phi_{[r]}=0.
\]

The action of $P_t$ on a monomial is
\begin{equation}
P_tx^n
=\sum_{k=0}^{\lfloor n/2\rfloor}
\frac{n!}{(n-2k)!k!}\left(\frac t2\right)^kx^{n-2k}.
\label{eq:heatmonomial}
\end{equation}
Because $\phi_{[1]}$ is supported on powers $1\pmod7$, substitution of
$x=t^4y$ in \eqref{eq:heatmonomial} gives
\begin{equation}
v_{[1]}(t,t^4y)=t^4W_1(t^7,y),
\qquad
W_1(0,y)=y-\frac{\rho}{16}.
\label{eq:W1}
\end{equation}
Similarly, the powers $3\pmod7$ give
\begin{equation}
v_{[3]}(t,t^4y)=t^5W_3(t^7,y),
\qquad
W_3(0,y)=\frac y2-\frac{\rho}{32}.
\label{eq:W3}
\end{equation}
Both limiting functions have the same simple zero $y=\rho/16$.

\begin{theorem}[Stationary-mode correspondence]
\label{thm:correspondence}
Let $v_{[1]}$ and $v_{[3]}$ be defined by \eqref{eq:heatmodes}.  There are
unique germs $F_{[1]}$ and $F_{[3]}$, analytic in the graded variable $t^7$,
such that
\[
v_{[r]}(t,F_{[r]}(t))=0,
\qquad
F_{[r]}(t)=\frac{\rho}{16}t^4+O(t^{11}).
\]
They coincide with the two determinant branches:
\begin{equation}
F_{[1]}=F_1,
\qquad
F_{[3]}=F_2.
\label{eq:correspondence}
\end{equation}
In particular, the first analytic boundary branch is generated by the
residue-$1$ stationary solution and the second by the residue-$3$ stationary
solution.
\end{theorem}

\begin{proof}
Equations \eqref{eq:W1} and \eqref{eq:W3}, together with
\[
\partial_yW_1(0,\rho/16)=1,
\qquad
\partial_yW_3(0,\rho/16)=\frac12,
\]
allow the analytic implicit-function theorem to be applied at
$(t^7,y)=(0,\rho/16)$.  It gives unique analytic functions $y_1(t^7)$ and
$y_3(t^7)$ satisfying $W_r(t^7,y_r(t^7))=0$.  Hence
\[
F_{[r]}(t)=t^4y_r(t^7)
\]
are convergent graded boundary germs.

Since $v_{[r]}$ satisfies both the heat equation and $N_4v_{[r]}=0$, each
boundary satisfies the determinant compatibility equation.  Direct
coefficient extraction from \eqref{eq:heatmonomial} and
\eqref{eq:stationary-rec} gives
\begin{align*}
F_{[1]}(t)
&=\frac{\rho}{16}t^4+\frac{\rho^3}{10240}t^{11}+O(t^{18}),\\
F_{[3]}(t)
&=\frac{\rho}{16}t^4+\frac{\rho^3}{12288}t^{11}+O(t^{18}).
\end{align*}
Under the universal scaling \eqref{eq:scaling}, these corrections are
$\Phi'(0)=1/5$ and $\Phi'(0)=1/6$, respectively.  The uniqueness assertion
in Theorem~\ref{thm:two} now gives \eqref{eq:correspondence}.
\end{proof}

\begin{remark}[The remaining local mode]
The absorbing condition removes the residue-$0$ stationary solution, leaving
the local directions represented by $x$, $x^2$, and $x^3$.  The residue-$2$
solution begins with $x^2$.  Its heat transform begins with $x^2+t$, so its
small zeros have square-root behavior $x\sim\pm\sqrt{-t}$ rather than a Taylor
expansion in $t$.  The restriction to graded analytic boundaries therefore
selects the residue-$1$ and residue-$3$ modes and excludes the residue-$2$
mode for a structural reason.
\end{remark}

%%%%%%%%%%%%%%%%%%%%%%%%%%%%%%%%%%%%%%%%%%%%%%%%%%%%%%%%%%%%
\section{Discussion}
%%%%%%%%%%%%%%%%%%%%%%%%%%%%%%%%%%%%%%%%%%%%%%%%%%%%%%%%%%%%

The cubic construction of \cite{HerGuerra} contained the seed of the
stationary-mode mechanism: a stationary ODE was heat-transported and its
boundary compatibility produced the exact curve $\rho(t/2)^3$.  The present
analysis does not replace that result.  It identifies the particular residue
mode used there and explains why its grading is preserved under heat
transport.  The fourth-order problem makes the mechanism visible because two
different stationary modes now lead to two different analytic branches.

The calculation exhibits a sequence of reductions:
\[
\text{boundary-jet system}
\longrightarrow
\det\M[F]=0
\longrightarrow
\E[\Phi]=0
\longrightarrow
\theta\mathbf y=\mathbf H(z,\mathbf y).
\]
The large determinant is not itself the final object.  Its grading reveals a
universal singular equation, and the singularity then yields a finite
algebraic branch condition.  Once the branch is selected, the remaining
calculation becomes triangular.

This phenomenon should not yet be interpreted as a theorem for all orders.
The number of branches, the correct scaling, and the nature of the
desingularized system may depend delicately on the sparse operator.  The
fourth-order result is useful precisely because all of these features can be
proved explicitly: there are two normalized analytic germs, both converge,
and their coefficients are computable by an exact recurrence.

The low-order picture can therefore be summarized as follows:
\[
\begin{array}{c|c|c}
\text{operator}&\text{stationary mode}&\text{transported boundary class}\\
\hline
\D^2+\rho x&x+\cdots&\rho(t/2)^2\\
\D^3+\rho x^2&x+\cdots&\rho(t/2)^3\\
\D^4+\rho x^3&x+\cdots&F_1\\
\D^4+\rho x^3&x^3+\cdots&F_2
\end{array}
\]
In both the cubic and fourth-order problems, the stationary mode beginning
with $x^2$ produces square-root rather than Taylor behavior.  This table is an
overall picture of the mechanism, not a conjectural classification theorem.

\begin{remark}[Moving zeros and hitting-time densities]
The existence of a transported solution satisfying
\[
v(t,F(t))=0
\]
does not by itself construct a first-hitting-time density.  The stationary
modes considered here have initial trace $v(0,x)=\phi(x)$, whereas a killed
heat kernel for Brownian motion must satisfy a Dirac initial condition,
together with the relevant positivity, normalization, and decay properties.
Only for such a kernel may its boundary derivative be identified, up to the
orientation convention, with the hitting-time density.

The quadratic case has an additional spectral feature that explains why the
probabilistic reconstruction in \cite{HerGuerra} is possible.  For
\[
L_2=c_{2,0}\D^2+c_{0,0}+c_{0,1}x,
\]
the coefficient $c_{0,0}$ does not enter the minimal compatibility equation
that determines the quadratic boundary: its contribution is multiplied by
$v(t,F(t))=0$.  Thus a family of stationary modes indexed by $c_{0,0}$ shares
the same moving boundary, and a Sturm--Liouville superposition can be used to
recover the Dirac initial condition.  At orders three and four, however,
determining the boundary requires differentiated compatibility relations, in
which $\D(c_{0,0}v)=c_{0,0}v^{(1)}$ need not vanish.  The spectral parameter
may therefore alter the boundary, so an analogous common-boundary
reconstruction is not automatic.  The present results concern transported
zeros of individual stationary modes; constructing a killed fundamental
solution for the cubic or fourth-order boundary is a separate problem.
\end{remark}

Several questions remain within the $L_4$ problem itself.  One may study the
radii of convergence and first complex singularities of $\Phi_1$ and
$\Phi_2$, their behavior on the real axis, and the significance of the
nonintegral Briot--Bouquet eigenvalue $1/7$.  Another question is whether the
two nonlinear boundary branches admit a direct interpretation in terms of
the three-dimensional Fourier solution space of $L_4\phi=0$.  These are
natural continuations of the present calculation, but are not needed for the
local two-branch theorem.  More broadly, the stationary grading suggests a
program for examining higher orders one at a time.  At present, however, the
mathematically established content consists of the quadratic baseline, the
published cubic construction with its new stationary-mode interpretation,
and the two fourth-order analytic branches proved here.

\enlargethispage{4\baselineskip}

\end{document}